\documentclass[12pt,amscd]{amsart}
\usepackage[pagebackref=true,colorlinks=true,linkcolor=blue,citecolor=blue]{hyperref}

\newtheorem{thm}{Theorem}[section]
\usepackage[all]{xy}
\usepackage{graphicx}
\usepackage{amsmath,amsxtra,amssymb,latexsym, amscd,amsthm}
\usepackage{indentfirst}
\usepackage[mathscr]{eucal}
\usepackage{hyperref}
\hypersetup{colorlinks=true,linkcolor=red,filecolor=magenta,urlcolor=green}

\newtheorem{cor}[thm]{Corollary}
\newtheorem{lem}[thm]{Lemma}

\theoremstyle{definition}
\newtheorem{defn}[thm]{Definition}
 
\newtheorem{rem}[thm]{Remark}

\DeclareMathOperator{\Adm}{Adm}

\DeclareMathOperator{\depth}{depth}

\DeclareMathOperator{\supp}{supp}

\DeclareMathOperator{\reg}{reg}
\DeclareMathOperator{\pd}{pd}
\DeclareMathOperator{\dst}{dstab}

\def\ab {\mathbf a}

\begin{document}

\title[Depth of symbolic powers of cover ideals]{Admissible subgraphs and the depth  of symbolic powers of cover ideals of graphs}

\author[T. D. Dung]{Tran Duc Dung}
	\address{Thai Nguyen University of Sciences, Phan Dinh phung Ward, Thai Nguyen, Vietnam}
	\email{dungtd@tnus.edu.vn}

\author[N.T. Hang]{Nguyen Thu Hang }
\address{Thai Nguyen University of Sciences, Phan Dinh phung Ward, Thai Nguyen, Vietnam}
\email{hangnt@tnus.edu.vn}

\author{Thanh Vu}
\address{Institute of Mathematics, VAST, 18 Hoang Quoc Viet, Hanoi, Vietnam}
\email{vuqthanh@gmail.com}

\begin{abstract} Let $G$ be a simple graph. We introduce the notion of $t$-admissible subgraphs of $G$ and show how to use them to compute the depth of the $t$-th symbolic powers of the cover ideal of $G$. As an application, we prove that
\[
\depth\big(S/J(C_n)^{(t)}\big) = n - 1 - \left\lfloor \frac{tn}{2t+1} \right\rfloor
\]
for all $t \ge 2$ and $n \ge 3$, where $S = K[x_1,\ldots,x_n]$ and $J(C_n)$ is the cover ideal of the cycle on $n$ vertices.
\end{abstract}

\maketitle
\section{Introduction}
Let \( S = K[x_1, \ldots, x_n] \) be a polynomial ring over a field \( K \), and let \( I \) be a homogeneous ideal of \( S \). A classical theorem of Brodmann~\cite{Br} asserts that the sequence \( \depth(S/I^t) \) is eventually constant for all sufficiently large \( t \). The least positive integer \( t_0 \) such that
\[
\depth(S/I^t) = \depth(S/I^{t_0})
\quad \text{for all } t \geq t_0
\]
is called the \emph{index of depth stability} of \( I \) and is denoted by \( \dst(I) \).

When \( I \) is a squarefree monomial ideal, a result of Hoa and Trung~\cite{HT} shows that the symbolic depth function \( \depth(S/I^{(t)}) \) also stabilizes. The smallest index at which it stabilizes is called the \emph{index of symbolic depth stability} of \( I \). 

Given a squarefree monomial ideal \( I \), determining its (symbolic) limit depth and the index of (symbolic) depth stability is a difficult problem. When restricted to the class of edge ideals and cover ideals of graphs, the problem is better understood. We now recall the definitions of these two ideals associated with a finite simple graph. Let \( G \) be a graph with vertex set \( V(G) = \{1, \ldots, n\} \) and edge set \( E(G) \). The \emph{edge ideal} and the \emph{cover ideal} of \( G \), denoted by \( I(G) \) and \( J(G) \), respectively, are defined by
\[
I(G) = (x_i x_j \mid \{i,j\} \in E(G)), 
\qquad
J(G) = \bigcap_{\{i,j\} \in E(G)} (x_i, x_j).
\]

Trung~\cite{Tr} showed that the limit depth \( \depth(S/I(G)^t) \) is equal to the number of bipartite connected components of \( G \). Recently, Lam, Trung, and Trung~\cite{LTT} described the index of depth stability of \( I(G) \). Both the symbolic limit depth and the index of symbolic depth stability of edge ideals remain unknown; see~\cite{MTV2} for more information. On the other hand, for cover ideals, the symbolic limit depth is known by results of Hoa, Kimura, Terai, and Trung~\cite{HKTT}, and Binh, Hang, Hien, and Trung~\cite{BHHT} provided a sharp bound for the index of symbolic depth stability of \( J(G) \).

While the limiting behavior of the depth of powers of edge ideals and the depth of symbolic powers of cover ideals of graphs is fairly well understood, the intermediate values are known only for a very limited class of graphs~\cite{BC, MTV1, HHV}. For symbolic powers of cover ideals of graphs, the first nontrivial results were recently obtained by Dung, Hang, Nam, and Tam \cite{DHNT} for paths. We note that when \( G \) is a bipartite graph, it is known that \( I(G)^t = I(G)^{(t)} \) and \( J(G)^t = J(G)^{(t)} \)~\cite{SVV, HHT}. In this work, we compute the depth of symbolic powers of cover ideals of cycles.

\begin{thm}\label{main}
Let \( C_n \) be a cycle on \( n \) vertices. Then, for all \( t \ge 2 \), we have
\[
\depth\big(S / J(C_n)^{(t)}\big) = n - 1 - \left\lfloor \frac{tn}{2t+1} \right\rfloor.
\]
\end{thm}

To achieve this, we use Hochster's formula for depth \cite{H} to reduce the problem to computing the regularity of admissible subgraphs; see Section~\ref{sec_2} for more details. For cycles, we explicitly describe the \( t \)-admissible subgraphs in order to derive the main result.

\section{Admissible subgraphs}\label{sec_2}
In this section, we recall basic notation and results used throughout the paper. We then introduce the notion of \( t \)-admissible subgraphs and explain how to use them to compute the depth of symbolic powers of edge ideals of graphs. Throughout the paper, let \( K \) be a field, \( S = K[x_1, \ldots, x_n] \) a polynomial ring, and \( \mathfrak{m} = (x_1, \ldots, x_n) \) the \emph{maximal homogeneous ideal} of \( S \).

\subsection{Depth and Castelnuovo--Mumford regularity}

Let \( L \) be a nonzero finitely generated graded \( S \)-module. Let \( H_{\mathfrak{m}}^i(L) \) denote the \( i \)-th local cohomology module of \( L \) with support in \( \mathfrak{m} \). Then the depth and the regularity of \( L \) are defined by
\[
\depth(L) = \min \{ i \mid H_{\mathfrak{m}}^i(L) \neq 0 \},
\]
and
\[
\reg(L) = \max \{ j + i \mid H_{\mathfrak{m}}^i(L)_j \neq 0,\ \text{for } i = 0, \ldots, \dim(L),\ j \in \mathbb{Z} \}.
\]

Let \( I \) be a monomial ideal in \( S \). From~\cite{H}, we recall the concept of an \emph{associated radical ideal} of \( I \), defined as follows.

\begin{defn}
Let \( I \) be a monomial ideal in \( S \), and let \( u \) be a monomial not contained in \( I \). The radical ideal \( Q := \sqrt{I : u} \) is called an \emph{associated radical ideal} of \( I \). We denote the set of all associated radical ideals of \( I \) by \( \mathrm{assrad}(I) \).
\end{defn}

\begin{rem}\label{associatedradical}
Let \( I \) be a monomial ideal in \( S \). For each associated prime \( P \) of \( I \), by~\cite{HH}, there exists a monomial \( f \) such that \( P = I : f \). Hence, \( P \) is an associated radical of \( I \). Furthermore, \( \sqrt{I} = \sqrt{I : 1} \) is also an associated radical of \( I \).
\end{rem}

Hochster showed that the depth of monomial ideals can be computed via the depth of their associated radicals.

\begin{thm}[Hochster]\label{Hochsterformula}
Let \( I \) be a monomial ideal of \( S \). Then
\[
\depth(S/I) = \min \{ \depth(S/Q) \mid Q \text{ is an associated radical of } I \}.
\]
\end{thm}

\subsection{Graphs and their edge ideal and cover ideals }
We recall some basic notions from graph theory; for further details, see~\cite{BM}.
\begin{defn}
Let \( G \) be a simple graph with vertex set \( V(G) = \{1, \ldots, n\} \) and edge set \( E(G) \).

\begin{enumerate}
    \item A simple graph \( H \) is a subgraph of \( G \) if \( V(H) \subseteq V(G) \) and \( E(H) \subseteq E(G) \). It is an \emph{induced subgraph} of \( G \) if \( E(H) = E(G) \cap \big( V(H) \times V(H) \big) \).
    
    \item For a subset \( U \subseteq V(G) \), we denote by \( G[U] \) and \( G \setminus U \) the induced subgraphs of \( G \) on \( U \) and on \( V(G) \setminus U \), respectively.
    
    \item A path \( P_n \) on \( n \) vertices is the graph with vertex set \( V(P_n) = \{1, \ldots, n\} \) and edge set
    \[
    E(P_n) = \{ \{1,2\}, \ldots, \{n-1,n\} \}.
    \]
    
    \item A cycle \( C_n \) on \( n \) vertices is the graph with vertex set \( V(C_n) = \{1, \ldots, n\} \) and edge set
    \[
    E(C_n) = E(P_n) \cup \{ \{1,n\} \}.
    \]
    
    \item A \emph{forest} is a graph with no cycles. A \emph{tree} is a connected forest.
    
    \item A subset \( M \subseteq E(G) \) is called a \emph{matching} of \( G \) if no two edges in \( M \) share a common vertex. It is an \emph{induced matching} if the subgraph induced by the vertices of \( M \) has edge set exactly \( M \). The \emph{induced matching number} of \( G \), denoted by \( \nu(G) \), is the maximum size of an induced matching in \( G \).
\end{enumerate}
\end{defn}

\begin{lem}\label{lem_ind_path}
Let \( P_n \) be a path on \( n \) vertices. Then
\[
\nu(P_n) = \left\lfloor \frac{n-1}{3} \right\rfloor.
\]
\end{lem}

\begin{defn}
Let \( G \) be a simple graph with vertex set \( V(G) = \{1, \ldots, n\} \) and edge set \( E(G) \). The \emph{edge ideal} of \( G \) is defined by
\[
I(G) = (x_i x_j \mid \{i,j\} \in E(G)) \subseteq S.
\]

The \emph{cover ideal} of \( G \) is defined by
\[
J(G) = \bigcap_{\{i,j\} \in E(G)} (x_i, x_j).
\]
\end{defn}

By a result of Jacques \cite{J}, we have
\begin{lem}\label{lem_reg_cycle}
Let \( C_n \) be a cycle on \( n \) vertices. Then
\[
\reg I(C_n) = 1 + \left\lfloor \frac{n+1}{3} \right\rfloor.
\]
\end{lem}

\subsection{Associated radicals of symbolic powers of cover ideals and admissible subgraphs} In this section, we assume that $G$ is a hypergraph with vertex set $V(G)$ and edge set $E(G)$. Each edge $e$ of $G$ is a subset of $V(G)$, and no edge is properly contained in another; that is, for any two distinct edges $e, f \in E(G)$, neither $e \subsetneq f$ nor $f \subsetneq e$ holds. A hypergraph $H$ is a \emph{subhypergraph} of $G$ if $V(H) \subseteq V(G)$ and $E(H) \subseteq E(G)$.

\begin{defn}
Let $G$ be a hypergraph with vertex set $V(G) = \{1, \ldots, n\}$ and edge set $E(G)$. The \emph{edge ideal} and \emph{cover ideal} of $G$, denoted by $I(G)$ and $J(G)$, are defined by
\[
I(G) = (x_e \mid e \in E(G))
\quad \text{and} \quad
J(G) = \bigcap_{e \in E(G)} (x_i \mid i \in e).
\]
The $t$-th symbolic power of the cover ideal of $G$ is defined by
\[
J(G)^{(t)}
=
\bigcap_{e \in E(G)} (x_i \mid i \in e)^t.
\]
\end{defn}

For an exponent \( \mathbf{a} = (a_1, \ldots, a_n) \in \mathbb{N}^n \), we denote \( x^{\mathbf{a}} = x_1^{a_1} \cdots x_n^{a_n} \). The support of \( \mathbf{a} \) (and of the monomial \( x^{\mathbf{a}} \)) is defined by
\[
\supp(\mathbf{a}) = \{ i \in \{1, \ldots, n\} \mid a_i \neq 0 \}.
\]

\begin{lem}\label{lem_assrad}
Let $G$ be a hypergraph, and let $x^\ab$ be a monomial in $S$. Then
\[
\sqrt{J(G)^{(t)} : x^\ab} = J(H),
\]
where $H$ is the subhypergraph of $G$ with edge set
\[
E(H)
=
\bigl\{
\{i_1,\ldots,i_s\} \in E(G)
\;\big|\;
a_{i_1} + \cdots + a_{i_s} < t
\bigr\}.
\]
\end{lem}

\begin{proof}
We have
\begin{align*}
\sqrt{J(G)^{(t)} : x^\ab}
&=
\sqrt{
\left(
\bigcap_{\{i_1,\ldots,i_s\} \in E(G)}
(x_{i_1},\ldots,x_{i_s})^t
\right)
: x^\ab
} \\
&=
\bigcap
\Bigl\{
(x_{i_1},\ldots,x_{i_s})
\;\Big|\;
\{i_1,\ldots,i_s\} \in E(G)
\text{ and }
a_{i_1} + \cdots + a_{i_s} < t
\Bigr\} \\
&=
J(H).
\end{align*}
The conclusion follows.
\end{proof}

\begin{defn}
A nonempty subhypergraph \( H \) of \( G \) is said to be \( t \)-\emph{admissible} if there exists \( \mathbf{a} \in \mathbb{N}^n \) such that
\begin{enumerate}
    \item \( a_{i_1} + \cdots + a_{i_s} \ge t \) for all \( \{i_1,\ldots,i_s\} \in E(G) \setminus E(H) \),
    \item \( a_{i_1} + \cdots + a_{i_s} < t \) for all \( \{i_1,\ldots,i_s\} \in E(H) \).
\end{enumerate}
We denote by \( \Adm_t(G) \) the set of all \( t \)-admissible subhypergraphs of \( G \), and by \( \Adm_t^*(G) \) the set of all \( t \)-admissible subhypergraphs of \( G \) excluding \( G \) itself.
\end{defn}

\begin{lem}\label{lem_depth_admissible}
Let \( G \) be a simple hypergraph. Then
\[
\depth(S/J(G)^{(t)}) = n - \max \{ \reg(I(H)) \mid H \text{ is a } t\text{-admissible subhypergraph of } G \}.
\]
\end{lem}

\begin{proof}
By Theorem~\ref{Hochsterformula} and Lemma~\ref{lem_assrad}, we have
\[
\depth(S/J(G)^{(t)}) = \min \{ \depth(S/J(H)) \mid H \in \Adm_t(G) \}.
\]
Since \( J(G) \) is the Alexander dual of the edge ideal \( I(G) \), by a result of Terai~\cite{Te}, we have \( \pd(S/J(H)) = \reg(I(H)) \). The conclusion follows from the Auslander--Buchsbaum formula.
\end{proof}

\begin{cor}\label{cor_cycle_depth}
Assume that \( G \) is a cycle or a forest. Then
\[
\depth(S/J(G)^{(t)}) = n - \max \{ \reg (I(G)), 1 + \nu(H) \mid H \in \Adm_t^*(G) \}.
\]
\end{cor}

\begin{proof}
Since any proper subgraph of a cycle is a forest, the conclusion follows from Lemma~\ref{lem_depth_admissible} and~\cite[Theorem 4.7]{BHT}.
\end{proof}

\section{Depth of symbolic powers of cover ideals of cycles}
In this section, we analyze the \( t \)-admissible subgraphs of cycles and prove the main theorem. We first fix some notation. Let \( H \) be a nonempty proper subgraph of a cycle \( C_n \). Then the edge set \( E(H) \) can be written uniquely as a disjoint union of maximal sets of consecutive edges. Suppose that
\[
E(H) = B_1 \cup \cdots \cup B_r,
\]
where each \( B_i \) consists of consecutive edges. For each \( i \), let \( C_i \) denote the set of edges between \( B_i \) and \( B_{i+1} \), where we identify \( B_{r+1} \) with \( B_1 \).

By relabeling if necessary, we may assume that \( 1 \) is the first vertex of an edge in \( B_1 \). For each \( i \), suppose that the vertices of the edges in \( B_i \) are \( (b_i, b_i+1, \ldots, c_i) \). Then the vertices of the edges in \( C_i \) are \( (c_i, c_i+1, \ldots, b_{i+1}) \). In particular, we have \( c_i < b_{i+1} \) for \( i = 1, \ldots, r-1 \).

\begin{lem}\label{lem_admissible_cycle}
Let \( C_n \) be a cycle on \( n \) vertices. Suppose that \( H \) is a subgraph of \( C_n \) with
\[
E(H) = B_1 \cup \cdots \cup B_r,
\]
where each \( B_i \) consists of consecutive edges supported on the interval \( [b_i, c_i] \) for \( i = 1, \ldots, r \). Then \( H \) is \( t \)-admissible if and only if there exist integers \( u_i, v_i \in \{0, \ldots, t-1\} \) for \( i = 1, \ldots, r \) such that:
\begin{enumerate}
    \item If \( |B_i| = 1 \), then \( u_i + v_i < t \);
    \item If \( |C_i| = 1 \), then \( v_i + u_{i+1} \ge t \), where \( u_{r+1} = u_1 \).
\end{enumerate}
\end{lem}
\begin{proof}
First, assume that \( H \) is \( t \)-admissible. Let \( \mathbf{a} \in \mathbb{N}^n \) be a vector satisfying the admissibility conditions. For each \( i \), let \( u_i \) and \( v_i \) be the values of \( \mathbf{a} \) at the endpoints of the block \( B_i \). Then \( u_i, v_i \le t-1 \). By definition, if \( |B_i| = 1 \), then \( u_i + v_i < t \), and if \( |C_i| = 1 \), then \( v_i + u_{i+1} \ge t \).

Conversely, suppose that integers \( u_i, v_i \in \{0, \ldots, t-1\} \) satisfy the stated conditions. We construct a vector \( \mathbf{a} \in \mathbb{N}^n \) as follows. Set \( a_{b_i} = u_i \) and \( a_{c_i} = v_i \). For all other vertices \( j \in V(B_i) \setminus \{b_i, c_i\} \), set \( a_j = 0 \). For vertices \( j \in V(C_i) \setminus \{c_i, b_{i+1}\} \), set \( a_j = t \). One checks directly that \( \mathbf{a} \) satisfies the required inequalities, and hence \( H \) is \( t \)-admissible.
\end{proof}

\begin{defn}
Let \( H \) be a \( t \)-admissible subgraph of \( C_n \). A tuple \( (\mathbf{u}, \mathbf{v}) \in \mathbb{N}^r \times \mathbb{N}^r \) satisfying the conditions of Lemma~\ref{lem_admissible_cycle} is called a \emph{certificate} of \( H \).
\end{defn}
\begin{defn}
A sequence \( \mathbf{b} = (b_1, \ldots, b_r) \) is said to be \( t \)-realizable for \( C_n \) if there exists a \( t \)-admissible subgraph \( H \) of \( C_n \) such that \( |B_i| = b_i \) for all \( i = 1, \ldots, r \).
\end{defn}

\begin{lem}\label{lem_red_1}
Assume that \( \mathbf{b} = (b_1, \ldots, b_r) \) is \( t \)-realizable. If \( b_i \ge 4 \), then
\[
\mathbf{b}' = (b_1, \ldots, b_{i-1}, 1, b_i - 3, b_{i+1}, \ldots, b_r)
\]
is also \( t \)-realizable.
\end{lem}

\begin{proof}
Let \( H \) be a \( t \)-admissible subgraph of \( C_n \) with edge decomposition \( E(H) = B_1 \cup \cdots \cup B_r \). By Lemma~\ref{lem_admissible_cycle}, \( H \) admits a certificate \( (\mathbf{u}, \mathbf{v}) \in \mathbb{N}^r \times \mathbb{N}^r \).

Consider the subgraph \( H' \) of \( H \) with
\[
E(H') = B_1 \cup \cdots \cup B_{i-1} \cup B_i^1 \cup B_i^2 \cup B_{i+1} \cup \cdots \cup B_r,
\]
where \( B_i^1 \) consists of the first edge of \( B_i \) and \( B_i^2 \) consists of the last \( b_i - 3 \) edges of \( B_i \). Then
\[
E(C_n) \setminus E(H') = C_1 \cup \cdots \cup C_{i-1} \cup C_i' \cup C_{i+1} \cup \cdots \cup C_r,
\]
where \( C_i' \) consists of two edges removed from \( B_i \).

Define \( \mathbf{u}' \) and \( \mathbf{v}' \) by
\[
u'_j =
\begin{cases}
u_j & \text{for } j = 1, \ldots, i, \\
0 & \text{if } j = i+1, \\
u_{j-1} & \text{for } j = i+2, \ldots, r,
\end{cases}
\quad
v'_j =
\begin{cases}
v_j & \text{for } j = 1, \ldots, i-1, \\
0 & \text{if } j = i, \\
v_{j-1} & \text{for } j = i+1, \ldots, r.
\end{cases}
\]

Since \( |C_i'| = 2 \), it imposes no constraint on \( (\mathbf{u}', \mathbf{v}') \). The only new constraints arise from \( B_i^1 \) and \( B_i^2 \), and these are satisfied since \( v_i' = 0 \) and \( u_{i+1}' = 0 \). Hence, \( (\mathbf{u}', \mathbf{v}') \) is a certificate for \( H' \). By Lemma~\ref{lem_admissible_cycle}, \( H' \) is \( t \)-admissible.
\end{proof}

\begin{lem}\label{lem_red_2}
Assume that \( \mathbf{b} = (b_1, \ldots, b_r) \) is \( t \)-realizable. If \( 2 \le b_i \le 3 \), then
\[
\mathbf{b}' = (b_1, \ldots, b_{i-1}, 1, b_{i+1}, \ldots, b_r)
\]
is also \( t \)-realizable.
\end{lem}

\begin{proof}
Let \( H \) be a \( t \)-admissible subgraph of \( C_n \) with edge decomposition \( E(H) = B_1 \cup \cdots \cup B_r \). By Lemma~\ref{lem_admissible_cycle}, \( H \) admits a certificate \( (\mathbf{u}, \mathbf{v}) \in \mathbb{N}^r \times \mathbb{N}^r \).

Consider the subgraph \( H' \) of \( H \) with
\[
E(H') = B_1 \cup \cdots \cup B_{i-1} \cup B_i' \cup B_{i+1} \cup \cdots \cup B_r,
\]
where \( B_i' \) consists of the first edge of \( B_i \). Then
\[
E(C_n) \setminus E(H') = C_1 \cup \cdots \cup C_{i-1} \cup C_i' \cup C_{i+1} \cup \cdots \cup C_r,
\]
where \( C_i' \) consists of the remaining edges from \( B_i \) and \( C_i \).

Define \( \mathbf{u}' = \mathbf{u} \) and \( \mathbf{v}' \) by
\[
v'_j =
\begin{cases}
v_j & \text{if } j \neq i, \\
0 & \text{if } j = i.
\end{cases}
\]

Since \( |C_i'| \ge 2 \), it imposes no constraint on \( (\mathbf{u}', \mathbf{v}') \). The only new constraint comes from \( B_i' \), which is satisfied since \( u_i' + v_i' = u_i < t \). Hence, \( (\mathbf{u}', \mathbf{v}') \) is a certificate for \( H' \). By Lemma~\ref{lem_admissible_cycle}, \( H' \) is \( t \)-admissible.
\end{proof}
\begin{lem}\label{lem_realized_1}
Let \( H \) be a subgraph of \( C_n \) with
\[
E(H) = \{e_j, e_{j+2}, \ldots, e_{j+2k}\}.
\]
Then \( H \) is \( t \)-admissible provided that \( k \le t-1 \).
\end{lem}

\begin{proof}
Let \( \mathbf{u} = (0,1,\ldots,k) \) and \( \mathbf{v} = (t-1,t-2,\ldots,t-k-1) \). Then \( (\mathbf{u}, \mathbf{v}) \) is a certificate for \( H \).
\end{proof}

\begin{lem}\label{lem_realized_2}
Let \( H_1 = (e_{i_1}, \ldots, e_{i_s}) \) and \( H_2 = (e_{j_1}, \ldots, e_{j_t}) \) be subgraphs of \( C_n \) such that
\[
e_{i_1} < \cdots < e_{i_s} < e_{j_1} < \cdots < e_{j_t},
\]
with \( j_1 - i_s \ge 3 \) and \( n + i_1 - j_t \ge 3 \). If both \( H_1 \) and \( H_2 \) are \( t \)-admissible, then \( H_1 \cup H_2 \) is \( t \)-admissible.
\end{lem}

\begin{proof}
Let \( (\mathbf{u}_1, \mathbf{v}_1) \) and \( (\mathbf{u}_2, \mathbf{v}_2) \) be certificates of \( H_1 \) and \( H_2 \), respectively. Then \( \mathbf{u} = (\mathbf{u}_1, \mathbf{u}_2) \) and \( \mathbf{v} = (\mathbf{v}_1, \mathbf{v}_2) \) form a certificate for \( H_1 \cup H_2 \), since the distance between \( H_1 \) and \( H_2 \) is at least \( 2 \), so the new gaps impose no additional constraints.
\end{proof}

\begin{lem}\label{lem_realized}
Let \( t \ge 1 \) and \( n \ge 3 \) be integers, and set \( m = \left\lfloor \frac{tn}{2t+1} \right\rfloor \). Then the sequence \( \mathbf{b} = \mathbf{1}^m = (1, \ldots, 1) \in \mathbb{N}^m \) is \( t \)-realizable.
\end{lem}

\begin{proof}
First, note that if \( t \ge \left\lfloor \frac{n}{2} \right\rfloor + 1 \), then \( m = \left\lfloor \frac{n}{2} \right\rfloor \), and \( \mathbf{1}^m \) is realizable by Lemma~\ref{lem_realized_1}. Hence, we may assume that \( t \le \left\lfloor \frac{n}{2} \right\rfloor \).

Write \( m = tq + r \) with \( 1 \le r \le t \). By Lemma~\ref{lem_realized_1}, a chain of \( t \) alternating edges of \( C_n \) is realizable. We now show that a stack consisting of \( q \) such chains together with one chain of \( r \) alternating edges is realizable. This configuration contains \( q(2t-1) + (2r-1) \) edges. By Lemma~\ref{lem_realized_2}, this configuration is realizable provided that the chains are separated by at least two edges. Thus, it suffices to show that
\[
q(2t-1) + (2r-1) + 2q + 2 \le n,
\]
which is equivalent to
\[
q(2t+1) + 2r + 1 \le n.
\]

By the definition of \( m \), we have
\[
n \ge \frac{m(2t+1)}{t} = \frac{(tq + r)(2t+1)}{t} = q(2t+1) + 2r + \frac{r}{t}.
\]
Since \( n \) is an integer, it follows that \( n \ge q(2t+1) + 2r + 1 \), as required.
\end{proof}

\begin{lem}\label{lem_realized_4}
Let $H$ be a subgraph of $C_n$ with $E(H) = \{e_j, e_{j+2}, \ldots, e_{j+2k}\}$. Assume that $H$ is $t$-admissible. Then $k \le t-1$.
\end{lem}

\begin{proof}
Let $\mathbf{u}, \mathbf{v}$ be a certificate of $H$. Then we have
\[
u_i + v_i \le t-1 \quad \text{for } i = 1, \ldots, k+1,
\]
and
\[
v_i + u_{i+1} \ge t \quad \text{for } i = 1, \ldots, k.
\]
In particular, $u_{i+1} > u_i$ for all $i = 1, \ldots, k$. Since $u_1 \ge 0$ and $u_{k+1} \le t-1$, we deduce that $k \le t-1$.
\end{proof}
\begin{lem}\label{lem_inq_2}
Let \( t \ge 2 \) and \( n \ge 3 \) be integers. Then
\[
\left\lfloor \frac{n+1}{3} \right\rfloor \le \left\lfloor \frac{tn}{2t+1} \right\rfloor.
\]
\end{lem}

\begin{proof}
Write \( n + 1 = 3k + r \) with \( 0 \le r \le 2 \). It suffices to show that \( k(2t + 1) \le tn \). Indeed, we compute
\[
t(3k + r - 1) - k(2t + 1) = tk - k + tr - t.
\]
If \( r > 0 \), then the right-hand side is nonnegative. If \( r = 0 \), then \( k \ge 2 \), and hence \( tk - k - t \ge 0 \). The conclusion follows.
\end{proof}

We are now ready to prove the main theorem.

\begin{proof}[Proof of Theorem~\ref{main}]
By Corollary~\ref{cor_cycle_depth} and Lemma~\ref{lem_reg_cycle}, it suffices to show that
\[
\max \{ \nu(H) \mid H \in \Adm_t^*(C_n) \} = \left\lfloor \frac{tn}{2t+1} \right\rfloor.
\]
By Lemmas~\ref{lem_red_1}, \ref{lem_red_2}, and \ref{lem_realized}, it is enough to consider the case where \( H \in \Adm_t^*(C_n) \) with
\[
E(H) = B_1 \cup \cdots \cup B_r
\quad \text{and} \quad |B_i| = 1 \text{ for all } i = 1, \ldots, r.
\]
We must show that \( r \le  \left\lfloor \frac{tn}{2t+1} \right\rfloor \).

As in the proof of Lemma~\ref{lem_realized}, we may assume that \( t \le \left\lfloor \frac{n}{2} \right\rfloor \). Decompose \( B_1 \cup \cdots \cup B_r \) into \( q \) chains of alternating edges, where each chain is separated from the next by at least two edges. Let \( \ell_1, \ldots, \ell_q \) denote the lengths of these chains. We have 
\[
\ell_1 + \cdots + \ell_q = r, \qquad 
\sum_{i=1}^q (2\ell_i - 1) + 2q \le n.
\]
By Lemma \ref{lem_realized_4}, \( \ell_i \le t \). Hence, \( qt \ge r \). Therefore,
\[
n \ge 2\sum_{i=1}^q \ell_i + q = 2r + q \ge 2r + \frac{r}{t}.
\]
Equivalently, \( r(2t + 1) \le tn \), which implies \( r \le \left\lfloor \frac{tn}{2t+1} \right\rfloor \), as desired.
\end{proof}

\begin{rem}
By~\cite[Corollary 3.6]{FHV}, we have \( \depth(S/J(C_n)^t) = 0 \) for all \( t \ge 2 \) when \( n \) is odd. When \( n \) is even, it is known that \( J(C_n)^t = J(C_n)^{(t)} \). Hence, Theorem~\ref{main} also determines the depth of powers of the cover ideal \( J(C_n) \).
\end{rem}

\begin{rem}
The study of the depth of symbolic powers of cover ideals of graphs is closely related to the study of ordered matchings in graphs. We refer to~\cite{CV, BHHT, HKTT} for further information.
\end{rem}

  \section{Acknowledgment} 
This research was funded by the TNU-University of Sciences under the project code: CS2026-TN06-04. 

\vspace{1cm}
\noindent {\bf Data Availability} Data sharing is not applicable to this article as no datasets were generated or analyzed during the current study.

\noindent {\bf Conflict of interest} There are no competing interests of either financial or personal nature.

\end{document}